\documentclass{article}
\usepackage{amssymb,amsfonts,amsmath,amsthm,amsopn,amstext,amscd,latexsym,xy,mathdots,stmaryrd}
\usepackage{hyperref}
\theoremstyle{plain}
\input xy
\xyoption{all}
\newtheorem{theorem}{Theorem}[section]
\newtheorem{lemma}[theorem]{Lemma}

\newtheorem{prop}[theorem]{Proposition}

\newtheorem{cor}[theorem]{Corollary}
\newtheorem{definition}[theorem]{Definition}

\newtheorem{conj}[theorem]{Conjecture}

\newtheorem{Thm}[theorem]{Theorem}
\newtheorem{Lem}[theorem]{Lemma}
\newtheorem{Prop}[theorem]{Proposition}
\newtheorem{Cor}[theorem]{Corollary}
\newtheorem{Def}[theorem]{Definition}

\newtheorem{Que}[theorem]{Question}

\theoremstyle{remark}

\numberwithin{equation}{section}
\numberwithin{paragraph}{section}
\newtheorem{Rem}[theorem]{Remark}

\DeclareMathOperator{\Hom}{Hom}

\DeclareMathOperator{\Gal}{Gal}
\DeclareMathOperator{\Aut}{Aut}

\DeclareMathOperator{\ad}{ad}

\DeclareMathOperator{\lcm}{lcm}

\DeclareMathOperator{\Pic}{Pic}

\DeclareMathOperator{\Ind}{Ind}

\DeclareMathOperator{\End}{End}

\DeclareMathOperator{\Frob}{Frob}

\DeclareMathOperator{\Spec}{Spec}

\newcommand{\GL}{\mathrm{GL}}

\newcommand{\SL}{\mathrm{SL}}

\newcommand{\N}{\mathbb{N}}

\newcommand{\A}{\mathbb{A}}
\newcommand{\Q}{\mathbb{Q}}

\newcommand{\C}{\mathbb{C}}

\newcommand{\aQ}{\overline{\mathbb{Q}}}
\newcommand{\Z}{\mathbb{Z}}

\newcommand{\aFl}{\overline{\mathbb{F}_{\ell}}}

\newcommand{\aQl}{\overline{\mathbb{Q}_{\ell}}}

\def\eps{\epsilon}
\def\rhobar{ {\bar {\rho} } }

\newcommand{\ra}{\rightarrow}

\newcommand{\F}{{\mathbb F}}

\newcommand{\BA}{{\mathbb{A}}}

\newcommand{\BC}{{\mathbb{C}}}

\newcommand{\BF}{{\mathbb{F}}\,\!{}}
\newcommand{\BG}{{\mathbb{G}}}

\newcommand{\BN}{{\mathbb{N}}}

\newcommand{\BQ}{{\mathbb{Q}}}

\newcommand{\BZ}{{\mathbb{Z}}}

\newcommand{\Fm}{{\mathfrak{m}}}
\newcommand{\Fn}{{\mathfrak{n}}}

\usepackage[ansinew]{inputenc}
\usepackage[dvipsnames,usenames]{color}

\def\dirlim{\mathop{\vtop{\hbox{\rm lim}\vskip-8pt
        \hbox{\hskip1pt$\scriptstyle\longrightarrow$}\vskip-1pt}}}
\def\invlim{\mathop{\vtop{\hbox{\rm lim}\vskip-8pt
        \hbox{\hskip1pt$\scriptstyle\longleftarrow$}\vskip-1pt}}}

\DeclareMathOperator{\id}{{id}}

\newcommand{\ab}{\mathrm{ab}}
\newcommand{\cts}{\mathrm{cts}}
\newcommand{\res}{\mathrm{res}}
\newcommand{\dimA}{g_A}
\newcommand{\can}{\mathrm{can}}

\newcommand{\ch}{\mathrm{ch}}
\newcommand{\blank}{\hbox{\phantom{i}}}
\newcommand{\image}{\mathop{{\rm Im}}\nolimits}

\newcommand{\kernel}{\mathop{\rm Ker}\nolimits}
\def\longto{\longrightarrow}

\title{Number of irreducible mod $\ell$  rank 2 sheaves on curves over finite fields}

\author{Gebhard  B\"ockle\footnote{\textsc{Interdisciplinary Center for Scientific Computing, Universit\"at Heidelberg, Heidelberg, Germany.} \textit{Email address}: \texttt{boeckle@uni-hd.de}}\ \ and Chandrashekhar B.  Khare\footnote{\textsc{Department of Mathematics, UCLA, Los Angeles, USA.} \textit{Email address}: \texttt{shekhar@math.ucla.edu}}} 

\begin{document}
\maketitle

\begin{abstract}
Let $X$ be a smooth projective curve of genus $g$  over a finite field $\F_q$ of characteristic $p$.  Consider primes  $\ell$ different from $p$. We formulate some questions related to a well known counting formula of Drinfeld in \cite{Drinfeld}. Drinfeld counts  rank 2, irreducible $\ell$-adic sheaves
on $X_n=X \times_{\F_q} \F_{q^n}$ as $n$ varies. We would like to count rank 2,  irreducible  mod $\ell$ sheaves on $X_n$ as $n$ varies.  Drinfeld's  $\ell$-adic count gives an upper bound for the mod $\ell$ count. We conjecture that Drinfeld's count is the correct asymptotic for the count of rank 2, irreducible mod $\ell$ sheaves on $X_n$ as $n$ varies, and $(n,\ell)=1$.  The conjecture is an invitation to finding methods to construct ``many'' irreducible  rank 2, mod $\ell$ sheaves on $X_n$ as $n$ varies.

We produce a lower bound on the mod $\ell$ count, which is weaker than the one conjectured, by counting  ``dihedral'' mod $\ell$ sheaves. We  make a deformation theoretic study of the number of $\ell$-adic sheaves which lift the pull backs $\mathcal F_n$ to $X_n$,  of a given mod $\ell$ sheaf $\mathcal F_{n_0}$ on $X_{n_0}$, as we vary $n$  with $n_0|n$.
\end{abstract}

\section{Introduction} 

\subsection{$\ell$-adic representations of $\pi_1(X)$}

Let $X$ be a smooth projective curve of genus $g$  over a finite field $\F_q$ of characteristic $p$.   Consider primes  $\ell$ different from $p$. Let $\aQ$ be an algebraic closure of $\Q$ and we  fix an embedding $\iota_\ell:\aQ \hookrightarrow \aQl$. 
In Theorem 1 of  \cite{Drinfeld}, a formula is given for the cardinality  $T(X,2,q^n,\ell)$  of the set of  irreducible two dimensional $\ell$-adic representations   $\rho: \pi_1(\overline X) \ra \GL_2(\aQl)$ which are fixed by ${\rm Frob}_q^n$, where $\overline X$ is the base change of $X$ to an algebraic closure $\overline{\F_q}$ of $\F_q$ and $\pi_1$ denotes the fundamental group. 

Let $U(X)$  be the set of (isomorphism classes of)  irreducible representations $\pi_1(\overline X) \rightarrow \GL_2(\aQl)$. It has an action of the geometric Frobenius ${\rm Frob}_q$, and  $T(X,2,q^n,\ell)$ is the cardinality of  $U(X)^{{\rm Frob}_q^n}$, the subset of  fixed points of $U(X)$   under ${\rm Frob}_q^n$.  It is easy to see  \cite[1.2]{Deligne-Comptage} that the set   $U(X)^{{\rm Frob}_q^n}$ is in natural bijection with equivalence classes under twisting by continuous characters $\Gal(\overline{\F_q}/\F_{q^n})\rightarrow \aQl^*$  of isomorphism classes of    2-dimensional $\ell$-adic representations of $\pi_1(X_n)$, with $X_n$ the base change of $X$ to $\F_{q^n}$, which remain irreducible on restriction to  $\pi_1(\overline X) $.  

%

As  a corollary of Theorem 1 of \cite{Drinfeld}, Drinfeld deduces:

\begin{theorem}\label{Drin}
 If $g>1$, then  there exists an integer $k$, and  integers $m_1,\cdots,m_k$, and Weil numbers  $\mu_1,\cdots,\mu_k$   of weights $j$ in the range  $0\leq j/2 <4g-3$ such that $$T(X,2,q^n,\ell)=q^{(4g-3)n}+\sum_{i=1}^km_i\mu_i^n.$$
\end{theorem}

This is a Lefschetz like formula and suggests that one is actually counting $\F_{q^n}$-valued points of a smooth projective  variety ${M}_X$ of dimension $4g-3$ that is defined over $\F_q$. The number $4g-3$ is  suggestive as it is the dimension of the moduli space  of irreducible two-dimensional unitary representations of the fundamental group of a compact Riemann surface of genus $g>1$. Deligne in \cite[1.5 and Conj.~2.15]{Deligne-Comptage} suggests instead  that there may be a variety $M_X$ in characteristic 0 and a suitable endomorphism $V$ on its cohomology such that the Lefschetz fixed point point formula applied to the  action of $V^n$ on $\ell$-adic cohomology gives the formula of Theorem \ref{Drin}.

Drinfeld's method of proof of Theorem \ref{Drin} is to count instead everywhere unramified cuspidal automorphic representations of $\GL_2(\A_{X_n})$, where $\A_{X_n}$ is the adeles of the function field of $X_n$, via the trace formula.
A theorem of Drinfeld shows that an irreducible representation $\pi_1(X_n) \rightarrow \GL_2(\aQl)$ arises from an everywhere unramified cuspidal  automorphic representation of $\GL_2(\A_{X_n})$ with respect to the fixed embedding $\iota_\ell$. Thus  the number $T(X,2,q^n,\ell)$ is independent of $\ell$, and may be denoted by $T(X,2,q^n)$. 
 The leading term of Drinfeld's formula comes from the trivial (identity) term, i.e. the orbital integral of the identity,  of the geometric side of the trace formula.

\subsection{Mod $\ell$ representations of $\pi_1(X)$}

We would like to consider a similar counting problem for mod  $\ell$ sheaves.
Namely consider  cardinality  $\overline{T(X,2,q^n,\ell)}$  of  the set of (isomorphism classes of) irreducible two dimensional mod $\ell$ representations $\rhobar:\pi_1(\overline X) \ra \GL_2(\aFl)$
which are fixed  by ${\rm Frob}_q^n$. Thus $\overline{T(X,2,q^n,\ell)}$ is the cardinality of  $\overline{U(X)}^{{\rm Frob}_q^n}$, the subset  $\overline U(X)$ of irreducible representations $\pi_1(\overline X) \rightarrow \GL_2(\aFl)$ that are fixed by ${\rm Frob}_q^n$. We again see $\overline{T(X,2,q^n,\ell)}$  is the number of  equivalence classes under twisting by continuous characters $\Gal(\overline{\F_q}/\F_q)\rightarrow \aFl^*$  of (isomorphism classes) of    2-dimensional representations of $\pi_1(X_n) \ra \GL_2(\aFl)$ which remain irreducible on restriction to~$\pi_1(\overline X) $. 

We also denote by $T'(X,2,q^n,\ell)$  and by $\overline{T'(X,2,q^n,\ell)}$  the cardinality  of   equivalence classes of    2-dimensional $\ell$-adic and mod $\ell$ representations of $\pi_1(X_n)$, respectively,  that remain irreducible on restriction to  $\pi_1(\overline X) $,  under twisting by continuous characters $\pi_1(X_n) \rightarrow \aQl^*$ and $\pi_1(X_n) \rightarrow \aFl^*$, respectively.

%

\subsection{Lifting}

A. de Jong has shown in \cite{deJong}  that the deformation rings of mod $\ell$ representations of $\pi_1(X_n)$ that are absolutely irreducible when restricted to $\pi_1(\overline X)$ are finite flat complete intersections over $\Z_\ell$. This implies  a lifting result and  one has:

\begin{lemma} [{\cite[Thm.~1.3]{deJong}}]\label{Lem-Lifting}
Every irreducible two dimensional mod $\ell$ representations $\rhobar:\pi_1(\overline X) \ra \GL_2(\aFl)$ which is fixed  by ${\rm Frob}_q^n$ lifts to an (irreducible) representations $\rhobar:\pi_1(\overline X) \ra \GL_2(\aQl)$ which is fixed  by ${\rm Frob}_q^n$.

In particular, we have the inequality 
$$\overline{T(X,2,q^n,\ell)} \leq  T(X,2,q^n,\ell).$$
\end{lemma}

\subsection{Conjecture}

We make the following conjecture. It is premised on the expectation that there are   few congruences mod $\ell$ between representations  in $U(X)^{{\rm Frob}_{q}^n}$ as we vary $n$.

\begin{conj}\label{Conj-Main} We conjecture that:


$${\rm lim}_{n \ra \infty} {\overline{T(X,2,q^n,\ell)} \over {T(X,2,q^n,\ell)}}=1$$ where $n$ runs through all positive integers prime to $\ell$, and 

$${\rm lim}_{n \ra \infty} {\overline{T'(X,2,q^n,\ell)} \over T'(X,2,q^n,\ell)}=1$$ where $n$ runs through all positive integers.


\end{conj}
 
 The difference in the  two statements above  is due to the  fact that if $n$ is  allowed to be divisible by high powers of $\ell$, there are mod $\ell$ congruences between characters $\pi_1(X_n) \ra\aQl^*$. We refer the reader to formulas (\ref{Formula-AsymptoticNotInEll}) and  (\ref{Formula-AsymptoticInEll}) in Section~\ref{SubSec-Dihedral} where quotients describing the asymptotics of counting mod $\ell$ versus $\ell$-adic characters are made explicit in a comparable context (though with different constants).

 The difficulty of the conjecture lies in the fact that we do not have a method to produce good lower bounds on  $\overline{T(X,2,q^n,\ell)}$  as $n$ varies. In fact the conjecture may be viewed as an invitation/challenge  to finding a method to  construct ``many''  irreducible representations 
 $\rhobar:\pi_1(X_n) \ra \GL_2(\aFl)$, as $n$ varies. It might be that if  we vary over $n$ that are prime to $\ell$, the number of congruences between representations $\rho:\pi_1(X_n) \ra \GL_2(\aQl)$ is bounded independently of $n$.

  Towards this, given a $\rhobar:\pi_1(X_n) \ra \GL_2(\aFl)$ which remains irreducible on restriction to $\pi_1(\overline{X})$ we make a qualitative (deformation theoretic) study in \S \ref{deformations} of  the cardinality  of the inverse image of $\rhobar$ under the reduction map
$U(X)^{{\rm Frob}_q^m} \ra \overline{U(X)}^{{\rm Frob}_q^m}$, for varying $m \in \N$ such that $n|m$.
We show for example that if we vary through $m$ such that $m \over n$ is prime to $ \ell$ then this cardinality is bounded independently of $m$ (cf. Theorem \ref{Thm-PrimeToEllTower}). We also analyze the case when $m \over n$ is a power of $\ell$ where the answer is very different (see Theorem \ref{Thm-EllPowerTower} 
 for the precise statement). 
 
 In Theorem \ref{Thm-CountingDihedrals} of  \S \ref{SubSec-Dihedral} we count the number of dihedral representations in $\overline{U(X)}^{{\rm Frob}_q^m}$ for varying $m$ which gives a weak lower bound towards the conjecture. 
 
 We end the paper by comparing the situation with counting 2-dimensional mod $p$ representations 
 $\rhobar: \Gal(\overline\Q/\Q)
 \ra \GL_2(\aFl)$ that are irreducible, odd and unramified outside $\ell$ as studied computationally in \cite{Centeleghe}.

{\bf Acknowledgements:} G.B. would like to thank the University of California at Los Angeles for hospitality where parts of this work was carried out. G.B. was supported with the DFG programs FG 1920 and SPP 1489 (in a joint project with FNR Luxembourg).  C.K. would like to thank TIFR  and Universit\"at Heidelberg for their  hospitality, and acknowledge support from a Humboldt Research Award and NSF grants, during the work on the paper.

\section{ Asymptotic results on universal deformation rings}\label{deformations}

Let $\BF$ be a finite field of characteristic $\ell$. For any field $K$ let $G_K=\Gal(\overline K/K)$.

\begin{Lem}\label{Lem-PrimeToEllTower}
Let $G$ be a profinite group that satisfies Mazur's finiteness condition $\Phi_\ell$, let $\bar\rho\colon G \to \GL_m(\BF)$ be a continuous representation, and let $H\subset G$ be a normal open subgroup. Suppose that 
\begin{enumerate}
\item the canonical homomorphism $\BF\to \End_{\bar\rho(H)}(\BF^m)$ is an isomorphism,
\item the restriction map $H^1(G,\ad)\to H^1(H,\ad)$ is an isomorphisms,
\item the restriction map $H^2(G,\ad)\to H^2(H,\ad)$ is injective,
\end{enumerate}
For $?\in\{G,H\}$ denote by $R_?$ the universal deformation ring for deformations of $\bar\rho|_{?}$ in the sense of Mazur -- note that $R_?$ exists by assumption~1. Then the canonical homomorphism $\pi\colon 
R_H\to R_G$ is an isomorphism.
\end{Lem}
\begin{proof}
Observe first that by the injectivity in assumption~2, the map on mod $\ell$ cotangent spaces induced from $\pi$ is surjective. Because $R_H$ and $R_G$ are complete noetherian local with residue field $\BF$, it follows that $\pi$ is surjective. Denote by $I$ the kernel of $\pi$. 
If $I$ is non-zero, we choose an ideal $\Fn$ of $R_H$ contained in $I$ such that $I/\Fn\cong\BF$ and consider the induced short exact sequence
\begin{equation}\label{RHtoRG}
0\to I/\Fn\to R_H/\Fn\to R_G\to0.
\end{equation}
There is an obstruction class $\theta$ in $H^2(G,\ad)$ to lifting the universal deformation $\rho_G\colon G\to \GL_n(R_G)$ from $R_G$ to $R_H/\Fn$. Upon restriction to $H$ this class vanishes, and so from assumption~3 we deduce $\theta=0$. It follows that the sequence (\ref{RHtoRG}) is split. But then the homomorphism on cotangent spaces induced from $\pi$ is not an isomorphism, contradicting assumption~2. Thus $I$ is zero and $\pi$ is an isomorphism.
\end{proof}

\begin{Rem}
If $\ell$ does not divide $m$, there is an obvious variant of Lemma~\ref{Lem-PrimeToEllTower} for deformations with fixed determinant, whose formulation, we leave to the reader.
\end{Rem}

\begin{Thm}\label{Thm-PrimeToEllTower} 
Let $\bar\rho\colon\pi_1(X)\to \GL_m(\BF)$ be a continuous representation that is absolutely irreducible when restricted to $\pi_1(\overline X)$. For each $n\ge1$ denote by $\rho_n\colon \pi_1(X_n)\longto \GL_m(R_{n})$ the universal deformation in the sense of Mazur of the restriction $\bar\rho|_{\pi_1(X_n)}$. Then there exists an $n_0$ prime to $\ell$ such that for any multiple $n$ of $n_0$ with $n$ prime to $\ell$, the canonical homomorphism $R_{n}\to R_{n_0}$ is an isomorphism.
\end{Thm}
\begin{proof}
We wish to apply Lemma~\ref{Lem-PrimeToEllTower} with $G= \pi_1(X_{n_0})$ and $H= \pi_1(X_{n})$ for $n_0|n$. Because $\ell$ is different from $p$, the group $\pi_1(\overline X)$ satisfies Mazur's condition $\Phi_\ell$, and hence so does  $\pi_1(X_n)$ for every $n$; this also implies that the groups $H^j(\pi_1(\overline X),\ad)$ below are finite-dimensional vector spaces over~$\BF$. Since $\bar\rho|_{\pi_1(X_n)}$ is absolute irreducible, it remains to verify assumptions 2 and 3 of Lemma~\ref{Lem-PrimeToEllTower}, for $n_0$ chosen suitably.

We begin with an application of the Hochschild-Serre spectral sequence to $0\to \pi_1(\overline X)\to \pi_1(X_n)\to G_{\BF_{q^n}}\to 1$. Since $G_{\BF_{q^n}}\cong\hat\BZ$ has cohomological dimension $1$, we obtain for all $j\ge0$ a short exact sequence 
\[ 0\to H^1(G_{\BF_{q^n}},H^j(\pi_1(\overline X),\ad)) \to H^{j+1}(\pi_1(X_n),\ad) \to H^0(G_{\BF_{q^n}},H^j(\pi_1(\overline X),\ad)) \to 0.\]
To have shorter formulas, we abbreviate $\Phi=\Frob_q$. Recall that for a $G_{\BF_{q^n}}$-module $M$ the cohomology $H^i(G_{\BF_{q^n}},M)$ for $i=0,1$ is isomorphic to the $G_{\BF_{q^n}}$ invariants and covariants of $M$ respectively. They are given as the $\Phi^n$ fixed points $M^{\Phi^n}$ of $M$ and the quotient $M/(1-\Phi^n):=M/(1-\Phi^n)M$, respectively. Restriction from $G_{\BF_{q^{n_0}}}$ to $G_{\BF_{q^n}}$ on cohomology is then described by the following explicit diagram
\[\xymatrix{
0\ar[r] & H^j(\pi_1(\overline X),\ad)/(1-\Phi^{n_0}) \ar[r]\ar[d]^\tau & H^{j+1}(\pi_1(X_{n_0}),\ad) \ar[r]\ar[d]^{\res} & H^j(\pi_1(\overline X),\ad)^{\Phi^{n_0}} \ar[r]\ar[d]^\iota & 0\\
0\ar[r] & H^j(\pi_1(\overline X),\ad)/(1-\Phi^n) \ar[r]& H^{j+1}(\pi_1(X_n),\ad) \ar[r]& H^j(\pi_1(\overline X),\ad)^{\Phi^n} \ar[r]& 0\rlap{,}\\
}\]
where $\iota$ is the inclusion homomorphism and $\tau$ is induced from the action of $1+\Phi^{n_0}+\ldots+\Phi^{n_0(n/n_0-1)}$ on $H^j(\pi_1(\overline X),\ad)$. The action of $\Phi$ on $H^j(\pi_1(\overline X),\ad)$ is a linear automorphism. Since $\ad$ is finite of characteristic $\ell$ we may choose $n'\ge1$ such that $\Phi^{n'}$ is unipotent. Let $n_0$ denote the minimal such choice. Suppose now that $n/n_0$ is not divisible by $\ell$. Then the number of summands of $\tau$ is of order prime to $\ell$, and since they are unipotent and commute with each other, we deduce that $\tau$ is an isomorphism. Observing further that $1-\Phi^{n}=(1-\Phi^{n_0})\tau$, we see that the kernels of $1-\Phi^n$ and $1-\Phi^{n_0}$ agree, and so also $\iota$ is an isomorphism. The assertion of Theorem~\ref{Thm-PrimeToEllTower} now follows from Lemma~\ref{Lem-PrimeToEllTower}.
\end{proof}
\begin{Rem}
The above theorem is an affirmative solution to a question of Hida raised over number fields. He asks whether in prime-to-$\ell$ (cyclotomic) towers the deformation rings of residual representations stabilize. This question is being investigated by Geunho Gim in his UCLA PhD thesis.
\end{Rem}
\begin{Cor}\label{Cor-PrimeToEllTower}
Suppose we are in the situation of Theorem~\ref{Thm-PrimeToEllTower}. Denote by $L_{\bar\rho,n}$ the set of lifts up to isomorphism of $\bar\rho$ to an $\ell$-adic representation. Then $L_{\bar\rho,n}$ is independent of $n$ as long als $n_0|n$ and $n/n_0$ is not divisible by $\ell$, i.e., restriction from $\pi_1(X_{n_0})$ to $\pi_1(X_{n})$ defines a bijection $L_{\bar\rho,n_0}\to L_{\bar\rho,n}$ for any such~$n$.
\end{Cor}
Another way to rephrase the corollary is to say that the number of congruences to a fixed $\bar\rho$ is constant for all multiples $n$ of $n_0$ as long as $n/n_0$ is of order prime to $\ell$.

\medskip

Theorem~\ref{Thm-PrimeToEllTower} concerns prime-to-$\ell$ towers. The following lemma prepares the analysis of the general situation.
\begin{Lem}\label{Lem-EllPowerTower}
Let $1\to H\to G\stackrel\pi\to Z\to 1$ be a short exact sequence of profinite groups such that $Z$ is procyclic, torsion free. Let $s\colon Z\to G$ denote a splitting of $\pi$. Suppose $H$ satisfies Mazur's finiteness condition $\Phi_\ell$. Let $\bar\rho\colon G \to \GL_m(\BF)$ be a continuous representation of degree $m$ prime to~$\ell$ such that the natural map $\BF\to \End_H(\BF^m)$ is an isomorphism.
For any closed subgroup $H'$ of $G$ containing $H$ denote by $R_{H'}^0$ the universal ring for deformations of $\bar\rho|_{H'}$ whose determinant is the Teichm\"uller lift of $\det\bar\rho|_{H'}$. Then
\begin{enumerate}
\item 
The canonical homomorphism $\can_{H,G}\colon R^0_H\to R^0_G$ is surjective.
\item 
Let $\rho_H\colon H\to \GL_m(R_H^0)$ be the universal deformation represented by $R_H^0$. For $\sigma\in Z$ define $\rho_H^\sigma\colon H\to \GL_m(R_H^0),h\mapsto \rho_H(s(\sigma) h s(\sigma)^{-1})$. There is a unique homomorphism $\alpha_\sigma\colon R_H^0\to R_H^0$ such that
\begin{equation}\label{Diag-CommUpToConj}
\xymatrix@C+2pc@R-1pc{
&\GL_m(R_H^0)\ar[dd]^{\alpha_\sigma}\\
H\ar[ur]^{\rho_H}\ar[dr]_{\rho^\sigma_H}&\\
&\GL_m(R_H^0)\rlap{.}\\
}
\end{equation}
commutes up to conjugation, and, if $I\subset R_H^0$ denotes the ideal generated by $\{r-\alpha_\sigma(r)\mid r\in R_H^0,\sigma\in Z\}$, then $R_H^0 \to R_G^0$ factors via $R^0_H/I$ and the induced map $R_H^0/I\to R_G^0$ is an isomorphism.
\item The canonical map $R_H^0\to \invlim_{U} R_{\pi^{-1}(U)}^0$ is an isomorphism where $U\!\subseteq \!Z$ ranges over all open subgroups.
\item If for all open subgroups $U\subset Z$ the ring $R_{\pi^{-1}(U)}^0$ is flat over $\BZ_\ell$ and reduced, then $\dirlim_U \Spec R_{\pi^{-1}(U)}^0[1/\ell]$ is Zariski dense in $\Spec R_H^0$, the rings $R_H^0[1/\ell]$ and $R_H^0$ are reduced, and $R_H^0$ is flat over $\BZ_\ell$.
\item \label{Part5}The conjugation action of $Z$ on $H/\kernel(\rho_H)$ via $s$ factors via a 
virtual pro-$\ell$ quotient~$\bar Z$ of~$Z$.
\item If the conjugation action of $Z$ on $H/\kernel(\rho_H)$ via $s$ is trivial, then $R_H^0\to R_G^0$ is an isomorphism.
\end{enumerate}
\end{Lem}

\begin{proof}
Let us first observe that indeed a continuous splitting $s\colon Z\to G$ exists. In the same way that one proves that the profinite completion $\hat\BZ$ is isomorphic to $\prod_{\ell'}\BZ_{\ell'}$ where $\ell'$ ranges over all prime numbers, one shows $Z\cong \prod_{\ell'}Z_{\ell'}$ for the procyclic group $Z$ where $Z_{\ell'}$ is the pro-$\ell'$ completion of $Z$. Because $Z$ is torsion free, each $Z_{\ell'}$ is either trivial or isomorphic to $\BZ_{\ell'}$. Denote by $g$ an element of $G$ whose image in $Z$ is a topological generator, and denote by $Z'$ the profinite completion of $g^\BZ$ in $G$. Then with similar notation as above we have 
$Z'\cong \prod_{\ell'}Z'_{\ell'}$, and moreover the surjective map $Z'\to Z$ is a surjection on each component. From the local shape of the $Z_\ell$, we see that we can choose $h$ inside $Z'\subset G$ such that the completion of $h^\BZ$ maps isomorphically to $Z$. 
The inverse is a continuous section as wanted. 

To prove part 1 it suffices to prove injectivity for the induced map on mod $\ell$ tangent spaces, i.e.\ for the restriction homomorphism $ H^1(G,\ad^0)\to H^1(H,\ad^0)$, which is 
part of the inflation restriction sequence
\begin{equation}\label{SES-InflRes}
 0\to H^1(G/H,H^0(H,\ad^0))\to H^1(G,\ad^0)\to H^1(H,\ad^0)^{G/H}\to H^2(G/H,H^0(H,\ad^0)).
\end{equation}
Because $\ell$ is prime to $m$, the isomorphism $\BF\to \End_H(\BF^m)$ yields $H^0(H,\ad^0)=0$, and part 1 follows.

We next prove part 2, which follows the proof of~\cite[Lem.~3.13]{deJong}. Let $\Fm_H$ denote the maximal ideal of $R_H^0$. For each $\sigma\in Z$ choose $A_\sigma\in\GL_m(R_H^0)$ whose reduction modulo $\Fm_H$ is equal to $\bar\rho(s(\sigma))$. Then $h\mapsto A_\sigma\rho^\sigma_H(h)A_\sigma^{-1}, H\to \GL_m(R_H^0),$ is a deformation of $\bar\rho|_H$ of determinant equal to the Teichm\"uller lift of $\det\bar\rho|_H$. Hence there exists a unique homomorphism $\alpha_\sigma\colon R_H^0\to R_H^0$ such that (\ref{Diag-CommUpToConj}) commutes up to conjugacy with $A_\sigma\rho^\sigma_H(\underline{\phantom{n}})A_\sigma^{-1}$ in place of $\rho_H^\sigma$. This proves the existence of $\alpha_\sigma$. Its uniqueness is a consequence of the universality of $R_H^0$. Observe that $\alpha_\sigma$ is an isomorphism as can be checked on tangent~spaces.

Next we show that the canonical map $\can_{H,G}\colon R_H^0\to R_G^0$ factors via $R_H^0/I$: denote by $\rho_G\colon G\to \GL_m(R_G^0)$ the universal deformation for $R_G^0$. Then for all $\sigma\in Z$ one has $\rho_G^\sigma=B_\sigma\rho_GB_\sigma^{-1}$ with $B_\sigma=\rho_G(s(\sigma))$. One deduces that $\can_{H,G}\circ\alpha_\sigma=\can_{H,G}$, which shows that $\can_{G,H}$ factors via $R_H^0/I$. The map induced on tangent spaces from $R_H^0/I\to R_G^0$ is described by the central arrow in (\ref{SES-InflRes}), which we know to be an isomorphism. To see that $R_H^0/I\to R_G^0$ is an isomorphy it thus suffices to 
construct a deformation $\hat\rho_H\colon G \to \GL_m(R_H^0/I)$ of $\bar\rho$ that extends $\rho_H$, since the latter yields a splitting of $R_H^0/I\to R_G^0$. 
For the details, we refer to the proof of \cite[Lem.~3.13]{deJong}. Alternatively, one can obtain $\hat\rho_H$ by extending the results of \cite[Subsec.~4.3.5]{HidaModFormsAndGalCoh} to a profinite setting: since $\bar\rho$ is given as a representation of $G$, the obstruction described in \cite[Thm.~4.35(5)]{HidaModFormsAndGalCoh} will lie in $H^2(Z,1+\Fm_H)$; the multiplicative group $1+\Fm_H$ is pro-$\ell$ and the mod $\ell$ cohomological dimension of the procyclic, torsion free group $Z$ is at most~$1$. 

For part 3, let $I_U$ denote the ideal of $R_H^0$ defined in item $2$ if one replaces $G$ by $\pi^{-1}(U)$. By part 2, it remains to prove that $\bigcap_{U}I_U=0$ where the intersection is over all open subgroups $U$ of $Z$. For this it suffices to show that for all $t>0$ there is a $U$ such that $\Fm_H^t  \supset  I_U$. To see the latter, we consider the quotient
\[  \rho_H/\Fm_H^t \colon H \to\GL_m(R_H /\Fm_H^t).\]
This representation factors via a finite quotient, say $H_t$, of $H$. Because the action of $Z$ on $H_t$ via $s$ and conjugation is continuous, there is an open subgroup $U$ of $Z$ such that $U$ acts trivially on $H_t$. This implies that $\alpha_\sigma\mod {\Fm_t}$ is the identity for all $z\in U$ and hence that $I_U$ is contained in $\Fm_H^t$.

To see 4, let $J\subset R_H^0$ be the radical ideal such that the Zariski closure of $\dirlim_U \Spec R_{\pi^{-1}(U)}^0[1/\ell]$ in $ \Spec R_H^0$ agrees with $\Spec R_H^0/J$. Because each $R_{\pi^{-1}(U)}^0$ is reduced the kernel of the canonical map $R_H^0\to R_{\pi^{-1}(U)}^0[1/\ell]$ contains $J$. But the map factors via $R_H^0\to R_{\pi^{-1}(U)}^0\to R_{\pi^{-1}(U)}^0[1/\ell]$, and by flatness of $R_{\pi^{-1}(U)}^0[1/\ell]$ over $\BZ_\ell$, the ideal $J$ lies in the kernel of $R_H^0\to R_{\pi^{-1}(U)}^0$. From the proof of part 3, we deduce that $J=0$. This proves the density assertion as well as the reducedness of  $R_H^0$ and $R_H^0[1/\ell]$. The vanishing of $J$ also implies that $R_H^0\to R_H^0[1/\ell]$ is injective, and this proves the flatness of $R_H^0$ over $\BZ_\ell$.

To prove \ref{Part5}, let $P^\ell$ be the kernel of the homomorphism from $\kernel(\bar\rho|_H)$ 
to its pro-$\ell$ completion. Then $P^\ell$ is a characteristic subgroup of $H$, that is closed. The quotient $H/P^\ell$ is thus a virtual pro-$\ell$ group that satisfies Mazur finiteness condition $\Phi_\ell$; moreover it surjects onto $\kernel(\rho_H)/H$. By \cite[Ch.~5]{DixonDuSautoyMannSegal} the automorphism group $\Aut(H/P^\ell)$ is in a natural way a profinite group that is virtually pro-$\ell$. Conjugation by elements of $Z$ via $s$ defines a closed subgroup of $\Aut(H/P^\ell)$, and from this part \ref{Part5} is clear.

For the last assertion it suffices to show that under the hypothesis of 6 one has $\alpha_\sigma=\id_{R_H^0}$ for all $\sigma\in Z$ for the maps in 3. This is clear, since the hypothesis implies that $\rho_H^\sigma=\rho_H$ for all $\sigma\in Z$.
\end{proof}

For $n\ge1$ define $\BF_{q^{n\ell^\infty}}=\dirlim_t \BF_{q^{n\ell^t}}$ and $X_{n\ell^\infty}=X\times_{\Spec\BF_q}\Spec \BF_{q^{n\ell^\infty}}$.
\begin{Thm}
\label{Thm-EllPowerTower} 
Suppose that the representation $\bar\rho\colon\pi_1(X)\to \GL_m(\BF)$ is absolutely irreducible when restricted to $\pi_1(\overline X)$. For $n\ge1$ and $t\in\{0,\infty\}$ denote by $\rho_{n\ell^t}\colon \pi_1(X_{n\ell^t})\longto \GL_m(R^0_{n\ell^t})$ the universal deformation in the sense of Mazur for deformations of the restriction $\bar\rho|_{\pi_1(X_{n\ell^t})}$ whose determinant is the Teichm\"uller lift of $\det\bar\rho|_{\pi_1(X_{n\ell^t})}$. Define similarly $\rho_{\infty}\colon \pi_1(\overline X)\longto \GL_m(R^0_\infty)$. Suppose that $\ell$ does not divide $m$ and that $\ell>2$ if $m>2$. Let $n\ge1$. Then the following hold:
\begin{enumerate}
\item The ring $R_n^0$ is reduced and finite flat over $\BZ_\ell$.
\item There is a canonical isomorphism $R_{n\ell^\infty}^0\to \invlim_t R_{n\ell^t}^0$, and $\dirlim_t \Spec R_{n\ell^t}^0$ is Zariski dense in $\Spec R_{n\ell^\infty}^0$.
\item There exists an $n_0$ such that for all multiples $n$ of $n_0$ one has a canonical isomorphism $R^0_\infty\to R_{n\ell^\infty}^0$.
\item If $g_{\overline X}$ denotes the genus of $\overline X$, then $R_\infty^0$ is formally smooth over $W(\BF)$ of relative dimension
\begin{equation}\label{Eqn-DimHone}
\dim H^1(\pi_1(\overline X),\ad^0)=(2g_{\overline X}-2)(m^2-1)=(2g_{\overline X}-2)\dim\ad^0.
\end{equation}
\end{enumerate}
\end{Thm}
\begin{proof}
Under the stated hypotheses on $\ell$, the finite flatness in part 1 follows from \cite[Thm.~1.3]{deJong} for $m\le2$ and \cite{Gaitsgory} for $m>2$. To see reducedness, observe that by finite flatness the ring $R_n^0[1/\ell]$ is a finite $\BQ_\ell$-algebra, and it suffices to show that it is a field. Let $\rho\colon \pi_1(X_n)\to\GL_m(\overline\BQ_\ell)$ be any irreducible representation. Then by \cite[Thm.~1]{Jannsen}, one has $H^2(\pi_1(X_n),\ad^0_{\rho})=0$ which in turn yields $H^1(\pi_1(X_n),\ad^0_{\rho})=0$, and hence that $\rho$ has no non-trivial deformations. This shows that $R_n^0[1/\ell]$ is a product of fields.

The proof of part 2 is a direct consequence of Lemma~\ref{Lem-EllPowerTower}, part 4, where for the short exact sequence $1\to H\to G\to Z\to 1$ one takes $1\to \pi_1(\overline X)\to  \pi_1(X)\to G_{\BF_q}\to1$. Next, by Lemma~\ref{Lem-EllPowerTower}, part 5, there exists $n_0>1$ such that $\Gal(\overline\BF_q/\BF_{n_0\ell^\infty})$ acts trivially on $\pi_1(\overline X)/\kernel(\rho_\infty)$. If one applies Lemma~\ref{Lem-EllPowerTower}, part 6, to the sequence $1\to \pi_1(\overline X)\to  \pi_1(X_{n\ell^\infty})\to \Gal(\overline\BF_q/\BF_{n\ell^\infty}) \to1$ for any $n$ with $n_0|n$, part 3 follows.

We finally prove part 4. The proof of the relative formal smoothness is given in \cite{deJong}. To compute the relative dimension one first observes that $\dim H^1(\pi_1(\overline X),\ad^0)$ is equal to the Euler-Poincar\'e characteristic of $\ad^0$ as a sheaf on $\pi_1(\overline X)$, because the relevant $H^0$ and $H^2$ terms vanish; the former by absolute irreducibility, the latter by Poincar\'e duality. The Euler-Poincar\'e characteristic is then evaluated to the number given in (\ref{Eqn-DimHone}) by the Theorem of Grothendieck-Ogg-Shafarevich, using that $\ad^0$ is lisse on~$\overline X$.
\end{proof}

\begin{Rem}
Let $\bar\rho\colon \pi_1(X)\to\GL_m(\BF_q)$ be as in the previous theorem. Because of part 1 of the theorem, for every $n$ and $t$, the ring $R^0_{n\ell^t}[\frac1\ell]$ is a product of fields. By the Theorem of Lafforgue, \cite{Lafforgue}, each factor corresponds to a cuspidal automorphic representation of $\GL_m(\BA_{X_{n\ell^t}})$. By parts 2--4, the number of such factors tends to $\infty$ as $t\to\infty$. This means that the number of cuspidal automorphic representations for $\GL_m(\BA_{X_{n\ell^t}})$ 
of fixed central character the Teichm\"uller lift of $\det\bar\rho$ that are congruent modulo $\ell$ to the given $\bar\rho$ also converges to $\infty$ for $t\to\infty$. However the theorem makes no quantitative assertions. For instance it sheds no light on the $\BQ_\ell$-dimension of $R^0_{n\ell^t}[\frac1\ell]$.
\end{Rem}



\section{Dihedral representations}
\label{SubSec-Dihedral}
In the following, we denote by $E$ an algebraically closed field of characteristic different from $2$ and from $p$. Unless indicated otherwise, $E$ will carry the discrete topology. We call a continuous representation $\rho$ of a profinite group $G$ over $E$ dihedral, if there exists an index $2$ subgroup $H$ of $G$ and a character $\chi\colon H\to E^*$ such that $\rho=\Ind_H^G\chi$ where $\chi$ satisfies 
$\chi\neq\chi^c$; here $c$ is any element in $G\setminus H$ and $\chi^c\colon H\to E^*$ is the character $h\mapsto \chi(chc^{-1})$ (which is independent of the chosen $c$). We note right away that $\rho$ is induced from exactly two such characters, namely $\chi$ and $\chi^c$. The condition $\chi\neq\chi^c$ implies that the image of $G$ under $\rho$ is non-abelian. Without further mentioning, we assume in the following that all representations are continuous. Because $E$ is discrete, all representations we consider have image of finite order.

\smallskip


The aim of this section is to deduce asymptotic formulas for the number of dihedral representations of $\pi_1(X_n)$, that remain dihedral after restriction to $\pi_1(\overline X)$, as $n$ varies, up to certain twist that we shall specify below. Such representations are induced from characters $\chi : \pi_1(Y_n) \rightarrow E^*$ for geometrically irreducible double covers $Y \rightarrow X$, and we shall count the latter up to twisting.
As recalled from \cite[1.2]{Deligne-Comptage} in the introduction, this number is the same as the number of dihedral representations of $\overline X$ that are fixed by the $n$-th power of the $q$-Frobenius. 
%



\subsection{$\ell$-power torsion of abelian varieties}
\label{SubSubSectEllPrimary}
Let $A$ be an abelian variety over $\BF_q$ of dimension $\dimA$. We collect some simple facts on the behavior of the size of the $\ell$-power torsion of $A(\BF_{q^n})$ for $n\to\infty$. The results below are inspired by the behavior of $|q^n-1|_\ell^{-1}$ for $n\to\infty$, where $q^n-1$ can be thought of as the cardinality of the $\BF_{q^n}$-valued points of the multiplicative scheme $\BG_m$. The results are probably well-known, but we could not locate a reference. For background on abelian varieties, we refer to the forthcoming book \cite{EdixhovenMoonenVanDerGeer}, and there in particular to Chapter~12.

By $\pi_A$ we denote the Frobenius of $A$ relative to $\BF_q$. We regard $\pi_A$ as an element of the $\BQ$-endomorphism ring $D=\End_\BQ(A)$. We write $\ch_A$ for the characteristic polynomial of $\pi_A$. It is the unique monic polynomial in $\BZ[T]$ of degree $2\dimA$ such that for all $n\in\BN$ one has $\ch_A(n)=\deg([n]_A-\pi_A)$, where $[n]_A$ is the multiplication by $n$ map on $A$ and $\deg$ the degree of an endomorphism, with the convention that the degree is zero if the endomorphism is not finite. For any prime number $\ell\neq p$, the polynomial $\ch_A$ is also the characteristic polynomial of the endomorphism on the $\ell$-adic Tate module of $A$ induced by $\pi_A$. The roots of $\ch_A$ in $\overline\BQ$ we denote by  $\alpha_i$, $i=1,\ldots,2\dimA$. They are $q$-Weil numbers (of weight one), i.e., algebraic integers $\alpha$ such that for any embedding $\iota\colon\overline\BQ\to\BC$ one has $|\iota(\alpha)|=q^{1/2}$. 

The elements of $A(\BF_{q^n})$ are equal to the kernel of the isogeny $\id-\pi_A^n\colon A\to A$. The degree of this isogeny is 
\[\# A(\BF_{q^n})= \deg(\id-\pi_A^n)= \prod_{i=1}^{2\dimA}(1-\alpha_i^n).\] 
Denote by $K$ the splitting field of $\ch_A$ over $\BQ$. We fix a prime $\ell$ different from $p$ and a place $\lambda$ of $K$ above~$\ell$. By $K_\lambda$ we denote the completion of $K$ at $\lambda$, by $k_\lambda$ the residue field of $K_\lambda$, by $\varpi_\lambda$ a uniformizer or $K_\lambda$ 
and by $|\blank|_\lambda$ the valuation on $K_\lambda$ that extends the valuation $|\blank|_\ell$ on $\BQ_\ell$ with $|\ell|_\ell=\frac1\ell$. The field $K_\lambda$ is unramified over $\BQ_\ell$ for almost all $\ell$. If this holds for $\ell$, one may take $\varpi_\lambda=\ell$. Choosing $\lambda$ enables us to analyze the factors $(1-\alpha_i^n)$ separately because $|\prod_i(1-\alpha_i^n)|_\ell=\prod_i|1-\alpha_i^n|_\lambda$. 

By the extension of Fermat's Little Theorem to finite fields we have $\alpha_i^{\#k_\lambda-1}\equiv1\pmod {\varpi_\lambda}$ for all $i=1,\ldots,2\dimA$. We denote by $h_\ell$ the smallest divisor of $\#k_\lambda-1$ such that 
$$\alpha_i^{h_\ell}\equiv1\pmod {\varpi_\lambda}\hbox{ for all }i=1,\ldots,2\dimA.$$
From the binomial theorem one deduces that $\alpha_i^{h_\ell\,\ell^j}\to 1$ for $j\to\infty$. 
Let $j_\ell\ge0$ be the smallest integer such that 
\[\big|  \alpha_i^{h_\ell\,\ell^{j_\ell}}-1 \big|_\lambda < \ell^{\frac{-1}{\ell-1}} <1 \hbox{ for all }i=1,\ldots,2\dimA 
.\]
Note that $j_\ell=0$ whenever $K/\BQ$ is unramified above $\ell$ and $\ell>2$. 

For each divisor $d$ of $h_\ell$ and each $j\ge0$ we define $N_{\ell,d,j}$ as $\# A[\ell^\infty](\BF_{q^{d\ell^{{}^j}}})$, and $g_d$ as the number of $i\in\{1,\ldots,2\dimA\}$ such that $\alpha_i^{d}\equiv1\pmod {\varpi_\lambda}$. Then $g_{h_\ell}=2\dim A\ge g_d$ for all $d$.

%
%
%
\begin{Prop}\label{PropEllPowerTorsion}
For $n\in\BN$ define $j=-\log_\ell |n|_\ell$, i.e., as the largest integer $j$ such that $\ell^j|n$. Then 
$$\#A[\ell^\infty](\BF_{q^n}) = \Big| \prod_{i=1}^{2\dimA}(1-\alpha_i^n) \Big|_\lambda^{-1}= N_{\ell,\gcd(n,h_\ell),j}.$$
Moreover for $j\ge j_\ell$ one has $N_{\ell,\gcd(n,h_\ell),j+k}=\ell^{kg_{\gcd(n,h_\ell)}}N_{\ell,\gcd(n,h_\ell),j}$.
\end{Prop}
Proposition~\ref{PropEllPowerTorsion} means that $\#A[\ell^m](\BF_{q^n}) $ depends only on the subgroup of $\BZ/(h_\ell)$ generated by~$n$ and the $\ell$-divisibility of~$n$, and the contribution of the $\ell$-divisibility behaves regular for $|n|_\ell^{-1}$ sufficiently~large.

\medskip

\begin{proof}[Proof of Proposition~\ref{PropEllPowerTorsion}]
The two points that require proof are:
\begin{enumerate}
\item For $j\ge0$ and $n\ge 1$ prime to $\ell$ one has $| 1-\alpha_i^{n\ell^j}|_\lambda=| 1-\alpha_i^{\gcd(n,h_\ell)\ell^j}|_\lambda$.
\item For $d$ a divisor of $h_\ell$ and $j\ge j_\ell$ one has $| 1-\alpha_i^{d\ell^{j}}|_\lambda=\big(\frac1\ell\big)^{j-j_\ell}\cdot | 1-\alpha_i^{d\ell^{j_\ell}}|_\lambda$.
\end{enumerate}
If $\alpha_i^n\not\equiv1\pmod {\varpi_\ell}$, then all four expressions above have the value $1$, and so the equalities are clear. So let us assume from now on that $\alpha_i^n\equiv1\pmod {\varpi_\ell}$. Then all four expressions take values in the open interval $(0,1)$; the value $0$ is impossible since $\alpha_i$ is a $q$ Weil number, and in particular not a root of unity. For equality 1 observe, that $n':=n/\gcd(n,h_\ell)$ is prime to $\ell$. Since $\alpha_i^{\gcd(n,h_\ell)\ell^j}\equiv1\pmod {\varpi_\ell}$ raising the element $\alpha_i^{\gcd(n,h_\ell)}$ to the power $n'$ does not change its distance to $1$, and this proves part~1.

Regarding equality 2, observe first, that by the just discussed part 1, we can on both sides replace $n$ by $h_\ell$ without changing the valuations. By the definition of $j_\ell$ we have 
$ | \alpha_i^{h_\ell\,\ell^{j_\ell}}-1 |_\lambda < \big(\frac1\ell\big)^{\frac1{\ell-1}}$. By induction on $j-j_\ell\ge0$ one easily deduces the equality in~2: for the induction one shows by the binomial theorem that $|(1+\eps)^\ell-1 |_\lambda = |(1+\ell\eps)-1|_\lambda = |\ell|_\ell |\eps|_\lambda$ if $\eps\in K_\lambda$ satisfies $|\eps|_\lambda<\big(\frac1\ell\big)^{\frac1{\ell-1}}$.
\end{proof}

\subsection{Counting dihedral representations up to twisting}
\label{Counting}

In this subsection, we fix the following abstract setting: $G$ is a profinite group and $H\subset G$ is an index $2$ subgroup such that one has commutative diagram with exact rows
\begin{equation}\label{DiagramForAbstDihedral}
\xymatrix{
0 \ar[r]& M'\ar[r]\ar[d]& H^\ab\ar[r]\ar[d]&\hat\BZ\ar[r]\ar@{=}[d]&0\\
0 \ar[r]& M\ar[r]& G^\ab\ar[r]^\kappa&\hat\BZ\ar[r]&0\rlap{,}\\
}
\end{equation}
in which $M$ and $M'$ are finite abelian groups. We shall give elementary formulas and estimates for the number of non-abelian dihedral representations $\rho$ over $E$ of $G$ induced from $H$ in terms of $M$ and $M'$. Recall that dihedral means that $\rho=\Ind_H^G\chi$, where $\chi\colon H\to E^*$ is a character such that $\chi\neq\chi^c$, and where $c$ is a fixed element of $G\setminus H$. 
By $\kappa$ we also denote the composite $G\to G^\ab\stackrel\kappa\to \hat\BZ$.

\begin{Def} Let $\rho$ and $\rho'$ be representations of $G$ (or of $H$) over $E$.


We call $\rho$ and $\rho'$ {\em strongly twist-equivalent over $G$ (over $H$)}, and write $\rho\approx\rho'$, if there is a character $\chi_0\colon G\to E^*$ (or $\chi_0\colon H\to E^*$) that is trivial on $\kernel \kappa$ (or $\kernel\kappa|_H$) such that $\rho'\cong\rho\otimes\chi_0$.

We call $\rho$ {\em stably irreducible} if the restriction $\rho|_{\kernel\kappa}$ is irreducible.
\end{Def}
The proof of the following lemma is straightforward and left to the reader.
\begin{Lem}\label{LemOnTwistEquiv-new1}
Let $\chi\colon H\to E^*$ be a character and set $\rho=\Ind_H^G\chi$.
\begin{enumerate}
\item The order of $\image(\rho)$ is invertible in $E$ and hence $\rho$ is semisimple.
\item 
The representation $\rho$ is stably irreducible if and only if $\chi|_{M'}\neq\chi^c|_{M'}$.
\item 
Suppose $\rho$ is stably irreducible. Then for any character $\chi'\colon H\to  E^*$ one has the following equivalences
\[\Ind_H^G\chi\approx \Ind_H^G\chi'  \hbox{ over }G\Longleftrightarrow \chi'\approx \chi \hbox{ or }\chi'\approx \chi^c \hbox{ over }H\Longleftrightarrow \chi'|_{M'}= \chi|_{M'} \hbox{ or }\chi'|_{M'}=\chi^c|_{M'}. \]
\item The restriction map $\chi\mapsto \chi|_{M'}$ defines a bijection
\[\{\hbox{characters }\chi\colon H\to E^*\}/\!\!\approx \ \to \{\hbox{characters }\chi_1\colon M'\to E^*\}. \]
\item 
The strong twist-equivalence classes of stably irreducible dihedral representations of $G$ over $E$ are in bijection with (unordered) pairs $\{\chi_1,\chi_1^c\}$ of characters $\chi_1\colon M'\to E^*$ such that $\chi_1\neq\chi_1^c$.
\end{enumerate}
\end{Lem}

One has a short exact sequence $1\to H^\ab\to G/[G,G]\to \BZ/2\to 1$. The element $c$ chosen above maps to the generator of $\BZ/2$, and conjugation by $c$ acts on $H^\ab$ as an endomorphism of order $2$. For any subgroup $H'\subset H^\ab$, we denote by $[c,H']$ the subgroup of $H^\ab $ generated by the commutators $[c,h']$, $h'\in H'$. Note that $[c,H^\ab]$ maps to $0$ under $\kappa$ and hence $[c,H^\ab]$ is a subgroup of~$M'$. We also denote by $M^{\prime-}$ the subgroup of $M'$ on which $c$ acts by multiplication with~$-1$. Again we leave the proof of the following result to the reader.
\begin{Lem}\label{LemOnTwistEquiv-new2}
Let $\chi_1\colon M'\to E^*$ be a character.
\begin{enumerate}
\item $\chi_1=\chi_1^c$ if and only if $[c,M']$ lies in the kernel of $\chi_1$.
\item $\chi_1$ extends to a character of $G/[H,H]$, i.e., of $M$, if and only if $[c,H^\ab]$ lies in the kernel of $\chi_1$. 
\item One has $2M^{\prime-}\subset [c,M^{\prime-}]\subset [c,M']\subset [c,H^\ab]\subset M^{\prime-}$. 
\item The index $e=[[c,H^\ab]:[c,M']]$ is $1$ or~$2$.
\item The map $\Hom(M,E^*)\to\{\chi_1\in\Hom(M',E^*)\mid \chi_1^c=\chi_1\}, \chi\mapsto \chi|_{M'}$ has kernel of order $2$ and image of index $e$, with $e$ from~4.
\end{enumerate}
\end{Lem}

By combining the above two lemmas, one proves the following:
\begin{Cor}
The number of strong twist-equivalence classes of stably irreducible dihedral representations of $G$ over $E$ is $\frac12(\#\Hom(M',E^*)-\frac{e}2\#\Hom(M,E^*))$, where $\Hom$ denotes homomorphisms of abelian groups.
\end{Cor}

Specializing the above to $E=\overline\BQ_\ell$ and $=\overline\BF_\ell$ yields the following result, where we use the duality between a finite abelian group and its group of characters, as well as~$\ell\neq2$.
\begin{Cor}\label{Cor-EllAdicAndModEllSystems}
The number of strong twist-equivalence classes of stably irreducible dihedral representations of $G$ over $\overline\BQ_\ell$ is $\frac12(\# M'-\frac{e}2 \#M)$.

The number of strong twist-equivalence classes of stably irreducible dihedral representations of $G$ over $\overline\BF_\ell$ is $\frac12(\# M'/M'_\ell-\frac{e}2\# M/M_\ell)$, where for a finite abelian group $A$ by $A_\ell$ we denote its $\ell$-primary part.
\end{Cor}

\medskip

The last results in this section concern the number of lifts to $\overline\BQ_\ell$ of a given dihedral representation of $G$ over $\overline\BF_\ell$. We write $\overline\BZ_\ell$ for the ring of integers of $\overline\BQ_\ell$. The following well-known result, based on the theorem of Brauer and Nesbitt, is relevant:
\begin{Prop}
Any representation $\rho\colon G\to\GL_n(\overline\BQ_\ell)$ is conjugate to a representation $\rho'$ whose image lies in $\GL_n(\overline\BZ_\ell)$. The semisimplification of the reduction of $\rho'$ to $\GL_n(\overline\BF_\ell)$ is independent of the chosen $\rho'$.
\end{Prop}
By $\bar\rho$ we will denote the semisimplification of a reduction of a representation $\rho\colon G\to\GL_n(\overline\BQ_\ell)$.

\smallskip

We remark that any character $\chi$ to $\overline\BQ^*_\ell$ can be written uniquely as a product $\chi=\chi_\ell\chi^\ell$, where $\chi_\ell$ has order a power of $\ell$ and $\chi^\ell$ has order prime to $\ell$. The proof of the following simple result on the reduction of characters and of dihedral representation is left to the reader.
\begin{Lem}
\begin{enumerate}
\item The map 
\[\{\chi\in \Hom(M',\overline\BQ_\ell^*)
\}
\longrightarrow \Hom(M',\overline\BF_\ell^*), \,\chi\longmapsto\bar\chi ,\]
is an epimorphism of abelian groups with kernel $\{\chi\in \Hom(M',\overline\BQ_\ell^*)\mid \chi=\chi_\ell \}$. It satisfies $\overline\chi=\overline{\chi^\ell}$.
\item The reduction of $\Ind_H^G\chi$ is $\Ind_H^G\bar\chi$.
The map
\[
\{
\Ind_H^G \chi\mid \chi\in \Hom_{\cts}(H,\overline\BQ_\ell^*) \}/\!\!\approx
\ \longrightarrow
\{
\Ind_H^G \bar\chi\mid \bar\chi\in \Hom_{\cts}(H,\overline\BF_\ell^*) \}/\!\!\approx,\, \rho\mapsto \bar\rho,
\]
between strong twist-equivalence classes is surjective. The fibers of stably irreducible representations $\Ind_H^G \bar\chi$ have cardinality $\#M_\ell$. 
\end{enumerate}
\end{Lem}

\begin{Rem}
Observe that the two numbers in Corollary~\ref{Cor-EllAdicAndModEllSystems} do not differ by a factor of $\#M_\ell$. This is so because the mod $\ell$ reduction of a dihedral representation may become a sum of characters.
\end{Rem}

\subsection{Growth of $\ell$-adic and mod $\ell$-systems of dihedral representations}
Suppose in the following that $\ell$ is a prime different from $2$, and that the genus $g$ of $X$ is at least~$2$.

We begin by recalling the main theorem of unramified geometric class field theory for $X$:
\begin{Thm}\label{ThmGeomCFT}
There is a commutative diagram
\[
\xymatrix{
0\ar[r]&\Pic^0(X)(\BF_q)\ar[d]^\simeq\ar[r]& \Pic(X)(\BF_q)\ar[d]^{[x]\mapsto \Frob_x}\ar[r]^-\deg& \BZ\ar[d]^{n\mapsto\Frob^n_{\BF_q}}\ar[r]&0 \\
0\ar[r]&\kernel(\res)\ar[r]& \pi_1^\ab(X)\ar[r]^-\res& G_{\BF_q}
\ar[r]&0 \rlap{,}
}\]
where the left vertical map is an isomorphism, the central and right vertical maps are injective with dense image, the central homomorphism is characterized by sending any (Weil) divisor $[x]$, for $x\in |X|$, to the well-defined Frobenius automorphism $\Frob_x\in\pi_1^\ab(X)$, and the homomorphism $\res$ is the restriction of the action of $\pi_1(X)$ to the unramified Galois pro-cover $\overline X$ of $X$ with group isomorphic to~ $G_{\BF_q}$.
\end{Thm}

We call a character of $\pi_1(X)$ a {\em character of the base} if it is the inflation of a character of $\pi_1(\overline X)$. Then the above theorem has the following immediate consequence.
\begin{Cor}
The homomorphism $ \Pic(X)(\BF_q)\to  \pi_1^\ab(X)$ of the previous theorem induces a bijection  
\[ \left\{
\begin{array}{c}
\mbox{finite order characters}\\
\Pic^0(X)(\BF_q)\to\BQ_\ell^*
\end{array}
\right\}
\longleftrightarrow
\left\{
\begin{array}{c}
\mbox{characters of }
\pi_1(X)^\ab\to\BQ_\ell^*\\
\mbox{up to twists by characters of the base.}
\end{array}
\right\}\]
\end{Cor}

\medskip

Our main concern are dihedral representations. 
Let $f\colon X'\to X$ be an unramified degree $2$ Galois cover of $X$ by a smooth projective curve $X'$ that is again geometrically connected over $\BF_q$. We write $\{1,c\}$ for $\Aut_X(X')$
. The automorphism $c$ induces an involution $c^*$ on the Jacobian $\Pic^0(X')$ of $X'$. 
To relate the above to the axiomatic setting of \ref{Counting}, we display the following commutative diagram with exact rows:
\[\xymatrix{
0 \ar[r]&  \Pic^0(X')(\BF_{q^n})\ar[r]\ar[d]& \pi_1(X'_n)^\ab\ar[r]\ar[d]&G_{\BF_{q^n}}
\ar[r]\ar@{=}[d]&0\\
0 \ar[r]&  \Pic^0(X)(\BF_{q^n})\ar[r]& \pi_1(X_n)^\ab\ar[r]&G_{\BF_{q^n}} 
\ar[r]&0\rlap{.}\\
}\]

\medskip

Before we state results on counting dihedral representation of $\pi_1(X_n)$ over $\overline\BQ_\ell$ and $\overline\BF_\ell$ for $n\to\infty$, we collect some basic results on the number of unramified degree $2$ Galois covers $f\colon X'_n\to X_n$, when $n$ varies: From Theorem~\ref{ThmGeomCFT} it follows that for each $n$ the set of degree $2$ unramified Galois covers of $X$ is in bijection with the index $2$ subgroups of 
$$\Pic^0(X)(\BF_{q^n})\otimes\BZ/2\oplus\Gal(\BF_{q^{2n}}/\BF_{q^{n}})\cong \pi_1^\ab(X_n)\otimes\BZ/2.$$
Let $X_n^{(i)} \to X_n$, $i=1,2$ be two such covers, and suppose that $\chi_i\colon \pi_1(X_n^{(i)})\to E^*$ are two characters, such that 
\[\Ind_{\pi_1(X_n^{(1)})}^{\pi_1(X_n)}\chi_1\approx\Ind_{\pi_1(X_n^{(2)})}^{\pi_1(X_n)}\chi_2.\] 
Then there is an isomorphism $X^{(1)}_{2n}\cong X^{(2)}_{2n}$ as coverings of~$X_n$. Hence if we wish to count unramified dihedral representation of $\pi_1(X_n)$ up to (strong) twist equivalence, then the relevant quadratic covers of $X_n$ are labelled by the index two subgroups of $\Pic^0(X)(\BF_{q^n})\otimes/\BZ/2$. The number of such labels is this equal to $\#\Pic^0(X)[2](\BF_{q^n})-1$.

To organize the above, for a non-trivial homomorphism $\beta\colon \pi_1(\overline X)\to\{\pm1\}$ we denote by $n_\beta$ the smallest $n$ such $\beta$ has an extension to $\pi_1(X_{n_\beta})\to\{\pm1\}$. Note that all $n$ such that $\beta$ extends to $\pi_1(X_{n})$ are multiples of $n_\beta$. By $f_\beta\colon X_{n_\beta}^{(\beta)}\to X_{n_\beta}$ we denote the corresponding unramified degree $2$ cover. The number of such $\beta$ is $2^{2g}-1$ where $g$ is the genus of $X$. 
For positive integers $n,m$ we define $\delta_{m|n}$  to be $1$ if $m$ divides $n$ and $0$ otherwise. Choose for each $\beta$ a number $e_\beta\in\{1,2\}$ according to Lemma~\ref{LemOnTwistEquiv-new2} part 4. 
Then we deduce from Corollary~\ref{Cor-EllAdicAndModEllSystems}:
\begin{Prop}
Let $\ell$ be an odd prime. The number of stably irreducible dihedral representation of $\pi_1(X_n)$ over $\overline\BQ_\ell^*$ up to twist by characters of the base is
\[ \sum_{\beta} \frac{\delta_{n_\beta|n} }2 \Big(\# \Pic^0(X_{n_\beta}^{(\beta)})(\BF_{q^n})-\frac{e_\beta}2 \#\Pic^0(X)(\BF_{q^n})\Big) .\]

The number of stably irreducible dihedral representation of $\pi_1(X_n)$ over $\overline\BF_\ell^*$ up to twist by characters of the base is
\[ \sum_{\beta} \frac{\delta_{n_\beta|n} }2 \Big(\# \frac{\Pic^0(X_{n_\beta}^{(\beta)})(\BF_{q^n})}{\Pic^0(X_{n_\beta}^{(\beta)})(\BF_{q^n})_\ell}-\frac{e_\beta}2 \frac{\#\Pic^0(X)(\BF_{q^n})}{\#\Pic^0(X)(\BF_{q^n})_\ell}\Big) .\]
\end{Prop}

From the Weil conjecture for curves we obtain sequences of real constants $(c_n)$ and $(c_{\beta,n})$ for $n\in\BN$ that satisfy $|c_n|\le q^{-n/2}$ and $|c_{\beta,n}|\le q^{-n/2}$, respectively, such that
\[ \#\Pic^0(X)(\BF_{q^n}) = q^{gn}(1-c_n)^{2g}\quad\hbox{and}\quad  \#\Pic^0(X_{n_\beta}^{(\beta)})(\BF_{q^n})  = q^{(2g-1)n}(1-c_{\beta,n})^{4g-2} .\]
From Proposition~\ref{PropEllPowerTorsion} we obtain 
\begin{itemize}
\item integers $h_\ell>0$ and $h_{\beta,\ell}>0$,
\item for every divisor $d$ of $h_\ell$ or $h_{\beta,\ell}$ integers $g_d\ge0$ and $g_{\beta,d}\ge0$ such that $g_{h_\ell}=2g>g_d$ for $d$ a proper divisor of $h_\ell$ and $g_{\beta,h_{\beta,\ell}}=4g-2>g_d$ for $d$ a proper divisor of $h_{\beta,\ell}$,
\item for any $d$ and $j\ge0$ let $N_{d,j}=\ell^{-g_dj}\#\Pic^0(X)[\ell^\infty](\BF_{q^{d\ell^j}})$ and $N_{\beta,d,j}=\ell^{-g_{\beta,d}j}\#\Pic^0(X_{n_\beta}^{(\beta)})(\BF_{q^{d\ell^j}})$,
\item some $j_\ell\ge0$ such that for all $j\ge j_\ell$ one has\footnote{The constants $N_{\ldots}$ as defined here differ by some $\ell$-powers from those in Proposition~\ref{PropEllPowerTorsion}.}
$$N_{d,j}=N_{d,j_\ell}\quad\hbox{ and }N_{\beta,d,j}=N_{\beta,d,j_\ell},$$
\end{itemize}
such that for any $n\ge0$ and $j(=j(n))=-\log_\ell |n|_\ell$ one has 
\[ \#\Pic^0(X)[\ell^\infty](\BF_{q^{n}})  = N_{\gcd(n,h_\ell),j} \ell^{g_dj} \quad\hbox{and}\quad  \#\Pic^0(X_{n_\beta}^{(\beta)})[\ell^\infty](\BF_{q^{n}})  = N_{\beta,\gcd(n,h_{\beta,\ell}),j} \ell^{g_{\beta,d}j}  .\]
\begin{Thm}\label{Thm-CountingDihedrals}
Let $\ell$ be an odd prime. 
\begin{enumerate}
\item \label{FirstBound}
The number of stably irreducible dihedral representation of $\pi_1(X_n)$ over $\overline\BQ_\ell^*$ up to twist by characters of the base is
\[ q^{(2g-1)n}\cdot \sum_{\beta} \frac{\delta_{n_\beta|n} }2\Big((1-c_{\beta,n})^{4g-2}- \frac{e_\beta}2 q^{-n(g-1)}(1-c_n)^{2g}\Big)  .\]
\item 
The number of stably irreducible dihedral representation of $\pi_1(X_n)$ over $\overline\BF_\ell^*$ up to twist by characters of the base is
\[ q^{(2g-1)n}
 \cdot \sum_{\beta} \frac{\delta_{n_\beta|n}}2  \Big(
\frac{(1-c_{\beta,n})^{4g-2}}{N_{\beta,\gcd(n,h_{\beta,\ell}),j}} q^{ -j  g_{\gcd(n,h_{\beta,\ell})} \log_q \ell}
-\frac{e_\beta}2 
\frac{(1-c_{n})^{2g}}{N_{\gcd(n,h_{\ell}),j}} q^{ -j  g_{\gcd(n,h_\ell)} \log_q \ell}
\Big) 
,\]
%
where as before $j=-\log_\ell|n|_\ell$ is the largest integer such that $\ell^j|n$ -- in particular $0\le j\le \log_\ell n$. 
\end{enumerate}
\end{Thm}
{
\begin{Rem}\label{Rem-Long}
We make some remarks that may help to better parse the expressions in Theorem~\ref{Thm-CountingDihedrals} parts 1 and 2: Let us note first that some of the constants introduced above depend on $\ell$. We did not want to add this extra term into the notation. The constants in formula 1, $c_n$, $n_\beta$, $e_\beta$, $c_{\beta,n}$ and $\tilde c_{\beta,n} $ do not depend on $\ell$. The constants $h_{\ldots}$, $g_{\ldots}$ and $N_{\ldots}$ from Proposition~\ref{PropEllPowerTorsion} do depend on $\ell$.

The expression in part~\ref{FirstBound} is independent of~$\ell$. The coefficient of its leading term $q^{(2g-1)n}$ is $\frac12\sum_{\beta} \delta_{n_\beta|n}$
. It depends on the non-vanishing of $\delta_{n_\beta|n}$, where $n_\beta$ is the minimal $n$, so that a certain $2$-torsion subgroup of $\Pic^0(X)$ is defined over $\BF_{q^{n_\beta}}$. There are $2^{2g}-1$ choices for $\beta$; say the set of such is $B$. For each divisor $d$ of $n_B=\lcm(n_\beta,\beta \in B)$ one obtains a different coefficient $C_d$ of the leading term. So asymptotically for $n\to\infty$ and fixed $d=\gcd(n,n_B)$, the expression in 1 behaves like 
$$C_{d} q^{n(2g-1)}(1+O(q^{-n/2})).$$ 
If $\min(n_\beta,\beta\in B)>1$, then $C_1=0$, and in fact the entire expression is zero for all $n$ such that $\gcd(n,n_B)=1$. I.e., within the so defined congruence classes, the expression is also zero asymptotically for $n\to\infty$. Note also the expression in part 1 is a Lefschetz like formula in the sense of Definition~\ref{Def-LefschetzLike}.

The second expression is more involved. There are finitely many structural constants 
since the $N_{d,j}$ and $N_{\beta,d,j}$ become constant for $j\to\infty$. This allows one to group all $n$ again into finitely many congruence classes (that depend on $n_B$, $h_\ell$ and $h_{\beta,\ell}$ and $j_\ell$) within which these structural constants are in fact constant. Within such a congruence class for $n$ (that is in fact of the form $\gcd(n,n_{B,\ell})=d$), one can write the expression in 2 as
\[ C'_d q^{(2g-1)(n-2j\log_q\ell)+j g'_d\log_q\ell} (1+O(q^{-n/2+c'_d\log_\ell n})),\]
where $j$ is the largest integer such that $\ell^j$ divides $n$. The quantity $g'_d$ lies between $0$ and $4g-2$, and for certain $d$ it does assume the value $4g-2$. This simply corresponds to the fact that (for all $\beta$ and all dihedral representations counted for that $\beta$) a given mod $\ell$ representation has roughly $\ell^{2j(2g-1-g'_d)}$ distinct lifts. 

If one fixes $j=j_0$ and considers $n\to\infty$ (within the fixed congruence class), then the growth in part 2 is slightly but consistently slower than that in part 1. Asymptotically the quotient of the two expressions is 
\begin{equation}\label{Formula-AsymptoticNotInEll}
C_d/C'_d \ell^{j_0(4g-2-g'_d)} 
\end{equation}
This is in line with Theorem~\ref{Thm-PrimeToEllTower} and Corollary~\ref{Cor-PrimeToEllTower}. 

If on the other hand one considers a sequence of the form $n_0\ell^k$ with $k\to\infty$, the quotient above takes the form
\begin{equation}\label{Formula-AsymptoticInEll}
C_d/C'_d q^{k(4g-2-g'_d)\log_q \ell} ,
\end{equation}
i.e., it grows logarithmically in $n$ in the exponent, or polynomially in $n$. This can be seen as a quantitative refinement of Theorem~\ref{Thm-EllPowerTower}.

The effect of changing $\ell$ does not qualitatively change the behavior of the expression in part 2 for $n\to\infty$. For $\ell\gg0$, the constants $N_{\ldots}$ simplify as explained before Proposition~\ref{PropEllPowerTorsion}: for all but finitely many $\ell$ one has $j_\ell=0$. However the quantities $h_{?,\ell}$ typically grow with $\ell$, and for each $\ell$, the $g_?$ will be new constants about which we cannot say anything.
\end{Rem}
}

\begin{cor}\label{dihedral}
Given any prime $\ell$ there are infinitely many $n$ such that $\overline{T(X,2,q^n,\ell)} < T(X,2,q^n,\ell)$.
\end{cor}

\begin{proof}
This is a consequence of Lemma \ref{Lem-Lifting} and  Theorem \ref{Thm-CountingDihedrals}.
\end{proof}

\section{Miscellaneous remarks}
We end with some remarks and by comparing the situation to counting mod $\ell$ eigenforms of level one.

\subsection{Lefschetz like formulas}

\begin{definition}\label{Def-LefschetzLike} By a  {\rm Lefschetz-like formula for the counting function $n\mapsto \overline{T(X,2,q^n,\ell)}$} we mean that there exist complex numbers $m_{i,\ell}$ and $\mu_{i,\ell}$ for $i$ between 1 and $k_\ell$, with $m_{i,\ell},\mu_{i,\ell}$ independent of $n$ but may be dependent on $\ell$, such that $\overline{T(X,2,q^n,\ell)}=\sum_{i=1}^{k_\ell} m_{i,\ell}\mu_{i,\ell}^n$ for all $n\ge1$.
\end{definition}

We show that there cannot be  a Lefschetz-like formula for $\overline{T(X,2,q^n,\ell)}$ for infinitely many $\ell$ with the number of Weil numbers  $k_\ell$ involved  in the formula bounded independently of $\ell$.  A Lefschetz like formula  of this sort should mean roughly that $\overline{T(X,2,q^n,\ell)}$ is counting the number of $\F_{q^n}$ points of  a variety $X_\ell$ which might depend on $\ell$ but whose dimension and dimensions of cohomologies  is bounded as $\ell$ varies.

\begin{prop}  
There is no infinite set $X$ of primes $\ell\neq p$, such that we have a Lefschetz like formula for  $n\mapsto\overline{T(X,2,q^n,\ell)}$ for all $\ell \in X$ with $k_\ell$ bounded independently of~$\ell$. 
\end{prop}

 \begin{proof}
Assume on the contrary that there is an infinite set $X$ which violates statement of the proposition. Let $k'\ge k+1$, with $k$ from Theorem~\ref{Drin} be a constant, such that for all $\ell\in X$ we have a Lefschetz like formula
\[\overline{T(X,2,q^n,\ell)}=\sum_{i=1}^{k_\ell} m_{i,\ell}\mu_{i,\ell}^n\quad\hbox{ for all }n\ge1,\]
for suitable complex numbers $m_{i,\ell}$, $\mu_{i,\ell}$, $i=1,\ldots,k'$, independent of $n$, but possibly dependent on $\ell$, such that $\mu_{1,\ell}$, \ldots, $\mu_{k_\ell,\ell}$ 
are pairwise distinct and non-zero, and $k_\ell\le k'$. The sequence $n\mapsto a_\ell(n):=\sum_{i=1}^{k_\ell} m_{i,\ell}\mu_{i,\ell}^n$ is what is called a generalized power sum. The numbers $m_{i,\ell}$ and $\mu_{i,\ell}$ are uniquely determined by this sequence; this follows from the asymptotics for the difference of two sequences for $n\to\infty$. 

As explained in \cite[\S~2]{vdPoorten}, the sequence $(a_\ell(n))_{n\ge0}$ is a recurrence sequence with minimal polynomial $s_\ell(X)=\prod_{i=1}^m(X-\mu_{i,\ell})$; and in fact the argument indicated there shows that $s_\ell(X)$ is the minimal polynomial of the sequence. In particular it is determined by $(a_\ell(n))_{1\le n\le 2k_\ell}$. Knowing the $\mu_{i,\ell}$, one can solve for the $m_{i,\ell}$ by using the values $(a_\ell(n))_{1\le n\le k_\ell}$, and hence these $k_\ell+k_\ell$ structural constants are determined by $(a_\ell(n))_{1\le n\le 2k_\ell}$. In particular, if we have $\overline{T(X,2,q^n,\ell)}= T(X,2,q^n,\ell)$ for all $n \leq 2k_\ell$, then from the above we deduce that $k_\ell=k+1$, $m_{i,\ell}=m_i$ and $\mu_{i,\ell}=\mu_i$ for $m_i,\mu_i$ as in Theorem~\ref{Drin}.

Now from de Jong's lifting result Lemma~\ref{Lem-Lifting} any representation in $\overline{T(X,2,q^n,\ell)}$ lifts to a representation in $T(X,2,q^n,\ell)$. Since clearly for any fixed $n$, there are only finitely many $\ell$ for which two distinct representation in $T(X,2,q^n,\ell)$ can have isomorphic reductions, there exists an $\ell_0$ such that for all $\ell>\ell_0$ and all $n\le 2k'$ we have the equality $\overline{T(X,2,q^n,\ell)}= T(X,2,q^n,\ell)$ required in the previous paragraph. But then the Lefschetz like formula with structural constants as in Theorem~\ref{Drin} shows that for almost all $\ell\in X$ we have $\overline{T(X,2,q^n,\ell)}= T(X,2,q^n,\ell)$ for all $n\ge1$. This however contradicts Corollary~\ref{dihedral} deduced from our explicit formulas on the number of dihedral representations. 
\end{proof}

The proposition and its proof  does not rule out the fact that  for $\ell\gg 0$, $\overline{T(X,2,q^n,\ell)}= T(X,2,q^n,\ell)$ for all $(n,\ell)=1$. 

\subsection{Analogy with mod $\ell$ modular forms}
\label{SubsecCent}
Let us recall an analogous situation studied by T. Centeleghe in \cite{Centeleghe}. He   counts the number  $N(G_\Q,2,\ell)$ of irreducible  mod $\ell$ representations $\rhobar:G_\Q \rightarrow \GL_2(\aFl)$ that are odd and unramified outside $\ell$. By Serre's conjecture these arise from newforms of level 1 and weight between 2 and $\ell+1$, up to twisting by powers of the mod $\ell$ cyclotomic character $\chi$. 

Thus  $N(G_\Q,2,\ell)$  is finite and is bounded above by $$(\ell-1)(\sum_{k=2}^{\ell+1} {\rm dim}_\C(S_k(\SL_2(\Z)))$$ and this number is roughly $\ell^3/48$. The conjecture here is that this is the correct asymptotic with $\ell$.
Something in fact stronger should be true to reflect the hypothesis that the number of congruences mod $\ell$ is negligible.

By an argument due to Serre, one can prove a lower bound like $\ell^2/8$, using roughly the fact that by bounds of Carlitz, at most $\ell/4$ of the Bernoulli numbers $B_k$ ($k$ between 2 and $\ell+1$) are divisible by $\ell$. For  weights $k$  with $2 \leq k \leq \ell+1$, and such that $\ell$ does not divide $B_k$, it is not hard to see that  all newforms in $S_k(\SL_2(\Z))$  give rise to irreducible mod $\ell$ representations. This together with the fact that at most two twists of  irreducible mod $\ell$ representations arising from cusp forms can have Serre weight  $\leq \ell+1$ shows that  $N(G_\Q,2,\ell) \geq \ell^2/48$.

Note that a more constructive proof using induced  dihedral representations may sometimes give no lower bounds when $\ell$ is 1 mod 4 (there are no unramified outside $\ell$, odd, 2-dimensional  mod $\ell$  dihedral representations when $\ell$ is 1 mod 4), or poor lower bounds like $\ell^{3/2}/2$  when $\ell$ is 3 mod 4 (as the class group of $\Q( \sqrt {(  - \ell ) }$ has order around $\sqrt{\ell}$). Serre's non-constructive proof does better, but still falls well short of the conjectured asymptotic.

\begin{Que}
Is there an analog of Serre's lower bound, and his style of argument,  in the present geometric case which improves the lower bound on $\overline{T(X,2,q^n,\ell)}$ (for fixed $\ell$ and varying $n$)  that we have obtained via counting dihedral representations in Theorem  \ref{Thm-CountingDihedrals}?
\end{Que}

\subsection{Congruences between cusp forms and Eisenstein series}

We pose a question related to producing irreducible liftings of reducible Galois representations $\rhobar:\pi_1(X) \rightarrow \GL_2(\F)$.

It is known that given any pair of characters $\chi_1,\chi_2:G_\Q \rightarrow \F^*$ such that $\chi_1\chi_2$ is odd,  there is a newform $f$ of  weight $k \geq 2$ and level $N$, the  prime to $\ell$ part of  the product of the  conductors of $\chi_1$ and $\chi_2$,  such that the $\ell$-adic representation $\rho_{f,\iota}:G_\Q \rightarrow \GL_2(\aQl)$ attached to $f$ has reduction whose semi simplification is $\chi_1 \oplus \chi_2$. This is proved by using  Ramanujan's  $\Theta$-operator on mod $\ell$ modular forms whose effect  on $q$-expansions of mod $\ell$ modular forms  is given by  $\Theta(\sum_n a_nq^n )=\sum_n na_nq^n$.  
The proof, cf. Theorem 1  of  \cite{Ghitza} for instance, is to start with a mod $\ell$ Eisenstein series $E=E_{\chi_1,\chi_2}$ which is an eigenform for Hecke operators $T_r$ for primes $ r$,  $(r,N\ell)=1$, with eigenvalue $\chi_1(r)+\chi_2(r)$ and observe that $f=\Theta^{\ell-1}E$ is a mod $\ell$ cuspidal form with the same Hecke eigenvalues. Note that  the weight of $f$ can be as large as $\ell^2-1$.

Motivated by this we can ask for analogs in our present function field case. Let $\chi:\pi_1(X) \ra \aFl^*$ be a character that factors thorough the Galois group of a geometric cover of~$X$.

Given  an irreducible $\ell$-adic representation $\rho:\pi_1(X) \rightarrow \GL_2(\aQl)$ such that  the semi simplification of  a reduction of it is given by $1\oplus \chi$. Then a standard argument using lattices and going back to Ribet \cite[Prop.~2.1]{Ribet} shows that $H^1(\pi_1(X),\aFl(\chi))$ and $H^1(\pi_1(X),\aFl(\chi^{-1}))$ are both non-zero. Conversely:

\begin{Que}
Suppose $H^1(\pi_1(X),\aFl(\chi))$ and $H^1(\pi_1(X),\aFl(\chi^{-1}))$ are both non-zero. Then  is there  an irreducible $\ell$-adic representation $\rho:\pi_1(X) \rightarrow \GL_2(\aQl)$ such that  the semisimplification of  a reduction of an integral model of it is given by $1\oplus \chi$?
\end{Que}

{
\bibliographystyle{hep}

\bibliography{counting}

%
%
%
%

\end{document}